\documentclass[a4paper,10pt]{amsart}
\usepackage{enumerate, amsmath, amsfonts, amssymb, amsthm, mathtools, thmtools, wasysym, graphics, graphicx, xcolor, frcursive,xparse,comment,ytableau,bbm,arydshln,epic,comment}
\usepackage{ mathrsfs }
\usepackage{multirow}
\usepackage{rotating}
\usepackage{appendix}
\usepackage[all]{xy}
\usepackage{hyperref}
\usepackage{csquotes}

\usepackage{url, hypcap}
\hypersetup{colorlinks=true, citecolor=darkblue, linkcolor=darkblue}

\usepackage{tikz}
\usetikzlibrary{calc,through,backgrounds,shapes,matrix}

\definecolor{darkblue}{rgb}{0.0,0,0.7} % darkblue color
\newcommand{\darkblue}{\color{darkblue}} % darkblue command
\definecolor{darkred}{rgb}{0.7,0,0} % darkred color
\definecolor{lightgrey}{rgb}{0.7,0.7,0.7} % darkred color
\usepackage{graphicx}

\def\CC{\mathbb{C}}

\newcommand{\hilb}{\mathcal{P}}

\newcommand{\spn}{{\mathrm{span}}}

\newcommand{\fix}{{\mathrm{fix}}}
\newcommand{\g}{{G}}
\newcommand{\n}{{N}}

\newcommand{\UN}{{U^N}}
\newcommand{\UNs}{{U^N_\sigma}}

\newcommand{\UNsd}{{(U^N_\sigma)^*}}

\newcommand{\Vs}{{V^\sigma}}

\newcommand{\Vd}{{V^*}}
\newcommand{\V}{{V}}

\newcommand{\ENs}{{\mathcal{E}^N_\sigma}}
\newcommand{\EN}{{\mathcal{E}_N}}

\newcommand{\h}{{H}}
\newcommand{\VN}{{E}}
\newcommand{\VNd}{{E^*}}

\newcommand{\fixed}{{\mathrm{fix}}}

\newtheorem{theorem}{Theorem}[section]
\newtheorem{proposition}[theorem]{Proposition}

\newtheorem{corollary}[theorem]{Corollary}
\newtheorem{lemma}[theorem]{Lemma}

\theoremstyle{definition}
\newtheorem{definition}[theorem]{Definition}
\newtheorem{example}[theorem]{Example}

\newtheorem{remark}[theorem]{Remark}

\usepackage[nameinlink,capitalize]{cleveref}
\usepackage[all]{xy}
\usepackage[T1]{fontenc}
\crefformat{footnote}{#2\footnotemark[#1]#3}
\crefformat{conjecture}{Conjecture~#2#1#3}
\crefformat{fact}{Fact~#2#1#3}
\AtBeginEnvironment{appendices}{\crefalias{section}{appendix}}

\DeclareFontFamily{U}{rcjhbltx}{}
\DeclareFontShape{U}{rcjhbltx}{m}{n}{<->rcjhbltx}{}
\DeclareSymbolFont{hebrewletters}{U}{rcjhbltx}{m}{n}
\DeclareMathSymbol{\lamed}{\mathord}{hebrewletters}{108}
\usepackage[colorinlistoftodos]{todonotes}

\newcommand{\defn}[1]{\emph{\darkblue #1}}
\usetikzlibrary{math}
\usepackage{xifthen}
\usepackage{xstring}
\newcommand\IfStringInList[2]{\IfSubStr{,#2,}{,#1,}}

\title[Normal Reflection Subgroups]
  {Normal Reflection Subgroups}

\author[C.~Arreche]{Carlos E. Arreche}
\address[C.~Arreche]{University of Texas at Dallas}
\email{arreche@utdallas.edu}

\author[N.~Williams]{Nathan F. Williams}
\address[N.~Williams]{University of Texas at Dallas}
\email{nathan.f.williams@gmail.com}

\date{\today}

\begin{document}
\maketitle

\begin{abstract}We study normal reflection subgroups of complex reflection groups. Our point of view leads to a refinement of a theorem of Orlik and Solomon to the effect that the generating function for fixed-space dimension over a reflection group is a product of linear factors involving generalized exponents. Our refinement gives a uniform proof and generalization of a recent theorem of the second author.\end{abstract}

\section{Introduction}

Hopf proved that the (singular) cohomology of a real connected compact Lie group $\mathcal{G}$ is an exterior algebra on $\mathrm{rank}(\mathcal{G})$ generators of odd degree~\cite{hopf1964topologie}.  Its Poincar\'e series is therefore given by \[\mathrm{Hilb}(H^*(\mathcal{G});q) = \prod_{i=1}^n (1+q^{2e_i+1}).\]
Chevalley presented these $e_i$ for the exceptional simple Lie algebras in his 1950 address at the International Congress of Mathematicians~\cite{chevalley1950betti}, and Coxeter recognized them from previous work with real reflection groups~\cite{coxeter1951product}.  This observation has led to deep relationships between the cohomology of $G$, and the invariant theory of the corresponding Weyl group $W=N_\mathcal{G}(T)/T$~\cite{reeder1995cohomology,reiner2019invariant}---notably, $H^*(\mathcal{G}) \simeq \left(H^*(\mathcal{G}/T) \times H^*(T)\right)^W \simeq 􏰁\left(S/S_+^W \otimes \bigwedge V^*\right)^W.$
For more information, we refer the reader to the wonderful survey~\cite{barcelo1994combinatorial}.

It turns out that one method to compute the $e_i$ is the generating function for the dimension of the fixed space $\fixed(w)=\dim(\ker(1-g))$ over the Weyl group:
\[\sum_{w \in W} q^{\fixed(w)} = \prod_{i=1}^r (q+e_i).\]

Shephard and Todd~\cite{shephard1954finite} verified case-by-case that the same sum still factors when $W$ is replaced by a complex reflection group $G$. Let $G$ be a finite complex reflection group, acting by reflections on $V$.  The $e_i$ are now determined by the degrees $d_i$ of the fundamental invariants of $G$ on $V$ as $e_i=d_i-1$.  A case-free proof of this result was given by Solomon~\cite{solomon1963invariants}, mirroring Hopf's result: writing $S=\mathrm{Sym}(V^*)$ and $\Lambda=\bigwedge(V^*)$, $(\left(S \otimes \Lambda\right)^G$ is a free exterior algebra over the ring of $G$-invariant polynomials, which gives a factorization of the Poincar\'e series 

\begin{equation}
\label{eq:solomon}
\mathrm{Hilb}(\left(S \otimes \Lambda \right)^G;q,u) = \prod_{i=1}^r \frac{1+u q^{e_i}}{1-q^{d_i}}.
\end{equation}
Computing the trace on $S \otimes \Lambda$ of the projection to the $G$-invariants $\frac{1}{G}\sum_{g\in G} g$, specializing to $u=q(1-x)-1$, and taking the limit as $x \to 1$ gives the Shephard-Todd result:

\begin{theorem}[{\cite{shephard1954finite,solomon1963invariants}}]\label{thm:OS-untwisted}
For any complex reflection group $G$,
\begin{equation}
\label{eq:shephard_todd}
\sum_{g \in G} q^{\fixed(g)} = \prod_{i=1}^r (q+e_i).
\end{equation}
\end{theorem}

More generally, define the \defn{fake degree} of an $m$-dimensional (simple) $G$-module $M$ to be the polynomial encoding the degrees in which $M$ occurs in the coinvariant ring $\mathbb{C}[V]_G$: $f_M(q)=\sum_{i} (\mathbb{C}[V]_G^i,M)q^i =\sum_{i=1}^m q^{e_i(M)}.$

\medskip
Let $G \subset \mathrm{GL}(V)$ be a complex reflection group.  We say that $N \triangleleft G$ is a \defn{normal reflection subgroup} of $G$ if it is a normal subgroup of $G$ that is generated by reflections.  For Weyl groups, nontrivial normal reflection subgroups can be constructed using root lengths.   More generally, normal reflection subgroups are constructed by taking the union of conjugacy classes of reflections.    We give the classification of normal reflection subgroups in~\Cref{sec:classification}, and tie our work with previous work on their numerology in~\Cref{sec:reflexponents}.

The following theorem is a special case of results in~\cite{bessis2002quotients} (where the authors consider the more general notion of ``bon sous-groupe distingu\'e'' in lieu of our normal reflection subgroup $N$ of $G$). We emphasize that our proof of this result in \Cref{sec:quotients} follows the main ideas in~\cite{bessis2002quotients}, specialized to our more restricted setting.

\begin{theorem}
\label{thm:reflection_action}
Let $N$ be a normal reflection subgroup of a complex reflection group $G$ acting on $V$.  Then the quotient group $H=G/N$ acts as a reflection group on the vector space $\VN=V/N$.\end{theorem}

In \Cref{sec:quotients} we will build on our proof of \Cref{thm:reflection_action} to prove the following numerological identities.

\begin{theorem}
\label{thm:numbers}
Let $G,N,H$ be as in~\Cref{thm:reflection_action}. For a suitable choice of indexing of degrees and fake degrees, we have the following equalities:
\begin{align*}
e_i^N(V) {+} e_i^G(\VN) &= e_i^G(V) \\
d_i^N \cdot d_i^H & = d_i^G \\
d_i^N \cdot e_i^H(\VN) & = e_i^G(\VN)
\end{align*}
\end{theorem}

\begin{example}
Take $G=W(F_4)=G_{28}$ and $N$ to be the normal subgroup generated by the reflections corresponding to short roots.  Then $N\simeq W(D_4)$, $G/N\simeq W(A_2)=\mathfrak{S}_3$ acting by reflections on $\CC \oplus \CC \oplus \CC^2$ (trivially on $\CC\oplus\CC$), so that \Cref{thm:numbers} gives the equations
\begin{align*}
(1,5,3,3) {+} (0,0,4,8) &= (1,5,7,11) \\
(2,6,4,4) \cdot (1,1,2,3) & = (2,6,8,12) \\
(2,6,4,4) \cdot (0,0,1,2) & = (0,0,4,8).
\end{align*}
\end{example}

The following result simultaneously generalizes the results of~\cite{williams2019reflexponents} and the Shephard-Todd formula~\Cref{eq:shephard_todd} from \Cref{thm:OS-untwisted}. We state a generalized version that incorporates Galois twists in \Cref{sec:galois}.

\begin{theorem}
\label{thm:main_theorem}
Let $N\triangleleft \ G$ be reflection groups acting by reflections on $V$, and let $\VN=V/N$. Then
\[\sum_{g \in G} q^{\fix_V g} t^{\fix_{\VN} g} = \prod_{i=1}^r \left(qt+e_i^N(V) t + e_i^G(\VN)\right).\]
\end{theorem}

\begin{example}
\label{ex:c2}
The dihedral group $G=G(2,1,2)=\left\langle s,t | s^2=t^2=(st)^4=1\right\rangle$ acts as a reflection group on $V=\mathbb{C}^2$ by $s=\left[\begin{smallmatrix} -1 & 0 \\ 0 & 1 \end{smallmatrix} \right] \text{ and } t=\left[\begin{smallmatrix} 0 & 1 \\ 1 & 0 \end{smallmatrix} \right].$  Take $N$ to be the normal subgroup generated by the reflections conjugate to $s$.   Then $N \simeq C_2 \times C_2$ is a normal reflection subgroup, isomorphic to the direct product of the cyclic group of order two with itself, with invariants $N_1=x_1^2$ and $N_2=x_2^2$, and $G$ acts on $\VN$ dual to $\VNd=\mathrm{span}_\CC\{N_1,N_2\}$ as the quotient group $G/N\simeq C_2$ by $s=\left[\begin{smallmatrix} 1 & 0 \\ 0 & 1 \end{smallmatrix} \right] \text{ and } t=\left[\begin{smallmatrix} 0 & 1 \\ 1 & 0 \end{smallmatrix} \right].$  In this case, \Cref{thm:main_theorem} expresses the equality
\[\sum_{g \in G} q^{\fix_V g} t^{\fix_{\VN} g} = q^2t^2+2qt^2+2qt+2t+t^2 =(qt+t)(qt+t+2).\]
\end{example}

\noindent
{\bf Acknowledgements.} We thank Theo Douvropoulos for many helpful comments and suggestions. The first author was partially supported by NSF grant CCF-1815108. The second author was partially supported by Simons Foundation award number 585380. 

\section{Quotients by Normal Reflection Subgroups}\label{sec:quotients} 

Let $V$ be a complex vector space of dimension $r$. A \defn{reflection} is an element of $\mathrm{GL}(V)$ that fixes some hyperplane pointwise. A \defn{complex reflection group} $G$ is a finite subgroup of $\mathrm{GL}(V)$ that is generated by reflections. A complex reflection group $G$ is called \defn{irreducible} if $V$ is a simple $G$-module; $V$ is then called the \defn{reflection representation} of $G$. A \defn{(normal) reflection subgroup} of $G$ is a (normal) subgroup of $G$ that is generated by reflections.

Let $S(V^*)$ be the symmetric algebra on the dual vector space $V^*$, and write $S(V^*)^G$ for its $G$-invariant subring.  By a classical theorem of Shephard-Todd~\cite{shephard1954finite} and Chevalley~\cite{chevalley1955invariants}, a subgroup $G$ of $\mathrm{GL}(V)$ is a complex reflection group if and only if $S(V^*)^G$ is a polynomial ring, generated by $r$ algebraically independent homogeneous $G$-invariant polynomials---the \defn{degrees} $d_1\leq \cdots \leq d_r$ of these polynomials are invariants of $G$.

\begin{theorem}[\cite{shephard1954finite,chevalley1955invariants}]
\label{thm:shephard_todd}
  Let $G \leq \mathrm{GL}(V)$ be finite.  Then $G$ is a complex reflection group if and only if there exist $r$ homogeneous algebraically independent polynomials $\g_1,\ldots,\g_r\in S(V^*)^G$ such that $S(V^*)^G=\CC[\g_1,\ldots,\g_r]$.  In this case, $|G|=\prod_{i=1}^r d_i$, where $d_i=\mathrm{deg}(\g_i)$.
\end{theorem}

Although \Cref{thm:reflection_action} is a special case of results in~\cite{bessis2002quotients} (where they consider the more general notion of ``bon sous-groupe distingu\'e''), the proof is more straightforward in our restricted setting, and also leads directly to a proof of \Cref{thm:numbers}.

{
\renewcommand{\thetheorem}{\ref{thm:reflection_action}}
\begin{theorem}
Let $N$ be a normal reflection subgroup of a complex reflection group $G$ acting on $V$.  Then the quotient group $H=G/N$ acts as a reflection group on the vector space $\VN=V/N$.
\end{theorem}
\addtocounter{theorem}{-1}
}

\begin{proof} We claim that there exist homogeneous generators $\n_1,\dots,\n_r$ of $S(V^*)^N$ such that $\VN^*=\mathrm{span}_\mathbb{C}\{\n_1,\ldots,\n_r\}$ is $H$-stable. By Theorem~\ref{thm:shephard_todd}, $S(V^*)^N=\mathbb{C}[\tilde{\n}_1,\dots,\tilde{\n}_r]$ for some homogeneous algebraically independent $\tilde{\n}_i$. Let $I_+\subset S(V^*)^N$ be the ideal generated by homogeneous elements of positive degree. Then both $I_+$ and $I_+^2$ are $H$-stable homogeneous ideals, and therefore the algebraic tangent space $I_+/I_+^2$ to $\VN=V/N$ at $0$ inherits a graded action of $H$ that is compatible with the (graded) quotient map $\pi:I_+\twoheadrightarrow I_+/I_+^2$. Hence there exists a graded $H$-equivariant section $\varphi:I_+/I_+^2\rightarrow I_+$. Letting $\n_i=\varphi\circ\pi(\tilde{\n}_i)$ we see that $\n_1,\dots,\n_r$ are still homogeneous algebraically independent generators for $S(V^*)^N$ with $\mathrm{deg}(\n_i)=\mathrm{deg}(\tilde{\n}_i)$ and such that $\VN^*=\spn_\CC\{\n_1,\dots,\n_r\}$ is $H$-stable, as claimed.

Write $\g_1, \ldots, \g_r$ for the homogeneous generators of $S(V^*)^G$, again as in Theorem~\ref{thm:shephard_todd}. Consider the action of $H$ on $\VN^*$ defined by $(gN)\n_i\coloneqq  g\n_i.$ Since $S(V^*)^G= (S(V^*)^N)^H=S(E^*)^H$, there exist polynomials $\h_1,\ldots,\h_r \in \CC[\mathbf{\n}]$ such that $\h_i(\mathbf{\n})=\g_i(\mathbf{x})$, where $\mathbf{\n}=\{\n_1,\dots,\n_r\}$ and $\mathbf{x}=\{x_1,\dots,x_r\}$ denote dual bases for $E$ and $V$, respectively.

Since any algebraic relation $f(\h_1,\ldots,\h_r)=0$ would result in an algebraic relation $f(\g_1,\ldots,\g_r)=0$, the $\h_i$ must be algebraically independent. By \Cref{thm:reflection_action}, $H$ is a complex reflection group.\end{proof}

\begin{remark}
Note that the quotient group $H=G/N$ does not necessarily lift to a reflection subgroup of $G$ nor even a subgroup of $G$. A counterexample is given by $G(4,2,2)=N \triangleleft G=G_8$, so that $G/N \simeq \mathfrak{S}_3$.
\end{remark}

\begin{remark}
\label{rem:indexing}
In the course of the proof of \Cref{thm:reflection_action} we showed that the vector space $\VN=V/N$ on which $H$ acts by reflections is dual to $\VNd:=\spn_\CC\{\n_1,\dots,\n_r\}$ for a certain choice of fundamental $N$-invariants $\n_1,\dots\n_r\in S(V^*)^N$ such that $\VNd$ is $G$-stable. The resulting action of $H$ on $\VN$ respects the $\mathbf{x}$-grading on the $N$-invariants $\n_i(\mathbf{x})$, and therefore there is a choice of fundamental $H$-invariants $\h_1,\dots,\h_r\in S(\VNd)^H=S(V^*)^G$ such that each $\h_i(\mathbf{N})=\g_i(\mathbf{x})$ is both $\mathbf{x}$-homogeneous of $\mathbf{x}$-degree $=:d_i^G$ and $\mathbf{N}$-homogeneous of $\mathbf{N}$-degree $=:d_i^H$. Since the action of $H$ on $\VNd$ respects the homogeneous decomposition of $\VNd$ according to $\mathbf{x}$-degree, the $\h_i(\mathbf{N})$ may be chosen such that the $N$-invariants $\n_j(\mathbf{x})\in\mathbf{N}$ occurring non-trivially in $\h_i(\mathbf{N})$ are all of the same $\mathbf{x}$-degree $=:d_i^N$. This relationship between the fundamental invariants for $N$, $G$, and $H$ acting on their corresponding reflection representations leads to interesting numerological identities.
\end{remark}

The following result motivates some of the theoretical ingredients in our proofs.

\begin{theorem}[{\cite{solomon1963invariants}}]
\label{thm:solomon} If $G\subset\mathrm{GL}(V)$ is a complex reflection group, then the ring $\left(S(V^*)\otimes \bigwedge V^*\right)^G$ is a free exterior algebra over the ring of $G$-invariant polynomials:
\[\left(S(V^*)\otimes \bigwedge V^*\right)^G \simeq S(V^*)^G \otimes \bigwedge U_G,\] where $U_G=\spn_\CC\left\{d\g_1,\ldots,d\g_s\right\}$ and $d\g_i=\sum_{j=1}^r \frac{\partial \g_i}{\partial x_j}  \otimes x_j$.\end{theorem}

\begin{comment}
\begin{proof}
Let $u_i:=d\n_i$. Consider the multiplication map $\mu:\UN\to \VN^*$ given by \[\mu(u_i)=\sum_{j=1}^rx_j\frac{\partial \n_i}{\partial x_j},\] and note that $\mu$ is $H$-equivariant of degree $+1$ and $\mu\circ d\n_i=d_i^N\n_i$.
\end{proof}
\end{comment}

{
\renewcommand{\thetheorem}{\ref{thm:numbers}}
\begin{theorem}
Let $G,N,H$ be as in~\Cref{thm:reflection_action}. For a suitable choice of indexing of degrees and fake degrees, we have the following equalities:
\begin{align*}
e_i^N(V) {+} e_i^G(\VN) &= e_i^G(V) \\
d_i^N \cdot d_i^H & = d_i^G \\
d_i^N \cdot e_i^H(\VN) & = e_i^G(\VN)
\end{align*}
%Let $G,N,H$ be as in~\Cref{thm:reflection_action}. For a suitable choice of indexing of degrees and fake degrees, we have:
%$e_i^N(V) {+} e_i^G(\VN) = e_i^G(V)$, $d_i^N \cdot d_i^H = d_i^G,$ and $d_i^N \cdot e_i^H(\VN) = e_i^G(\VN)$.
\end{theorem}
\addtocounter{theorem}{-1}
}

\begin{proof} Having chosen fundamental $N$-invariants $\n_1,\dots,\n_r\in S(V^*)^N$ such that $\VNd=\spn_\CC\{\n_1,\dots,\n_r\}$ is $G$-stable as in the proof of \Cref{thm:reflection_action} and \Cref{rem:indexing}, we have fundamental $G$-invariants $\g_i(\mathbf{x})=\h_i(\mathbf{N})$ that are $\mathbf{x}$-homogeneous of $\mathbf{x}$-degree $d_i^G$ and $\mathbf{N}$-homogeneous of $\mathbf{N}$-degree $d_i^H$, and where the $\n_j\in\mathbf{N}$ occurring non-trivially in $\h_i(\mathbf{N})$ are all of the same $\mathbf{x}$-degree $d_i^N$. The equality $d_i^Nd_i^H=d_i^G$ is immediate.

Let us show that this same choice of indexing of fundamental invariants for $N$, $G$, and $H$ results in the other two equalities. We begin by comparing $\mathbf{x}$-degrees in \[d\g_i=\sum_{j=1}^r\frac{\partial\g_i}{\partial x_j}\otimes x_j=\sum_{k=1}^r\frac{\partial \h_i}{\partial \n_k}\cdot d\n_k=\sum_{k=1}^r\sum_{j=1}^r\frac{\partial \h_i}{\partial \n_k}\cdot\frac{\partial\n_k}{\partial x_j}\otimes x_j.\] Recall that $e_i^G(V)=d_i^G-1=\mathrm{deg}_\mathbf{x}(d\g_i)$ and $e_i^N(V)=d_i^N-1=\mathrm{deg}_\mathbf{x}(d\n_i)$. Similarly, $e_i^H(E)=d_i^H-1=\mathrm{deg}_\mathbf{N}(d\h_i)$, where this time $d\h_i=\sum_{k=1}^r\frac{\partial \h_i}{\partial \n_k}\otimes\n_k\in  (S(\VNd)\otimes \VNd)^H$. Since $\frac{\partial \h_i}{\partial \n_k}=0$ whenever $\mathrm{deg}_\mathbf{x}(\n_k)\neq d_i^N$, it follows that \[e_i^G(V)=e_i^N(V)+d_i^N\cdot (d_i^H-1)=e_i^N(V)+d_i^N\cdot e_i^H(E).\] It remains to show that $d_i^Ne_i^H(E)=e_i^G(\VN)$.

The $e_i^G(\VN)$ are known to coincide with the $\mathbf{x}$-degrees of any set of homogeneous generators for $(S(V^*)\otimes\VNd)^G$ as a free $S(V^*)^G$-module. Since $\VNd$ consists of $N$-invariants, \[(S(V^*)\otimes\VNd)^G=((S(V^*)\otimes\VNd)^N)^H=(S(\VNd)\otimes\VNd)^H\simeq S(\VNd)^H\otimes U_H=S(V^*)^G\otimes U_H,\] where again $U_H:=\spn_\CC\{d\h_1,\dots,d\h_r\}$ and the non-trivial isomorphism comes from \Cref{thm:solomon} applied to the reflection representation $\VN$ of $H$. Hence $(S(V^*)\otimes\VNd)^G$ is generated by $d\h_i$ as a free $S(V^*)^G$-module, whence $e_i^G(\VN)=\mathrm{deg}_\mathbf{x}(d\h_i)=d_i^Ne_i^H(\VN)$.\end{proof}

\begin{remark}\label{rem:bigraded_solomon}
The same argument used in the proof of \Cref{thm:numbers} shows more generally: \[(S(V^*)\otimes\bigwedge \VNd)^G=((S(V^*)\otimes \bigwedge \VNd)^N)^H=(S(\VNd)\otimes\bigwedge \VNd)^H\simeq S(V^*)^G\otimes\bigwedge U_H.\] 
\end{remark}

\section{Poincar\'e Series and Specializations}\label{sec:poincare}

Our goal in this section is to prove our main result:

{
\renewcommand{\thetheorem}{\ref{thm:main_theorem}}
\begin{theorem}
Let $N\triangleleft \ G$ be reflection groups acting by reflections on $V$, and let $\VN=V/N$. Then
\[\sum_{g \in G} q^{\fix_V g} t^{\fix_{\VN} g} = \prod_{i=1}^r \left(qt+e_i^N(V) t + e_i^G(\VN)\right).\]
\end{theorem}
\addtocounter{theorem}{-1}
}

We refer to the left-hand side of~\Cref{thm:main_theorem} as the \defn{sum side}, and to the right-hand side as the \defn{product side}.  
We prove \Cref{thm:main_theorem} by computing the Poincar\'e series for $(S(V^*)\otimes\bigwedge \VNd)^G$ in two different (and standard) ways, keeping track of the supplemental grading afforded by the $\mathbf{x}$-degrees of $N$-invariants in $\VNd=\spn_\CC\{\n_1,\dots,\n_r\}$: one way corresponds to the product side (\Cref{sec:prod}), and the other to the sum side (\Cref{sec:sum}).  A subtlety arises when trying to compute the term-by-term specialization for the sum side, which is dealt with in~\Cref{sec:coseti,sec:cosetii}.

A more general version of~\Cref{thm:main_theorem} that incorporates  Galois twists is stated in~\Cref{sec:galois}. For technical reasons that arise in that generalization, we will define the shifted homogeneous decomposition $\VN^*_m:=\spn_\CC\{\n_i \ | \ \mathrm{deg}_\mathbf{x}(\n_i)=m+1\}$, and similarly
\[({\textstyle\bigwedge^p}\VNd)_m:=\spn_\CC\{ \n_{i_1}\wedge\dots\wedge\n_{i_p}\in{\textstyle\bigwedge^p}\VNd \ | \ {\textstyle\sum_{j=1}^p}\mathrm{deg}_\mathbf{x}(\n_{i_j})=m+p\}.\]
Writing $S(V^*)_\ell$ for the homogeneous component of $\mathbf{x}$-degree $\ell$, we define the Poincar\'e series
\begin{equation}\label{eq:p-series-defn}
\mathcal{P}(x,y,u):=\sum_{\ell,m,p\geq 0}\mathrm{dim}_\CC((S(V^*)_\ell\otimes ({\textstyle\bigwedge^p}\VNd)_m)^G)x^\ell y^m u^p.
\end{equation}
We write $\ENs=\{e_1^N(\Vs),\dots,e_r^N(\Vs)\}$ for the set of fake degrees of $\Vs$ as an $N$-representation.

\subsection{Product Side}
\label{sec:prod}
We first obtain the following product formula for the Poincar\'e series $\mathcal{P}(x,y,u)$ defined in \Cref{eq:p-series-defn} from \Cref{thm:numbers} and \Cref{rem:bigraded_solomon}, since $d\h_i\in S(V^*)_{e_i^G(E)}\otimes \VN^*_{e_i^N(V)}$.

\begin{lemma}
\label{lem:product_side}
$\hilb(x,y,u)=\displaystyle\prod\limits_{i=1}\limits^r\frac{1+x^{e_i^G(\VN)}y^{e_i^N(V)}u}{1-x_i^{d_i^G}}.$
\end{lemma}

\begin{corollary}
\label{cor:prod_side_specialization}
$\lim\limits_{x\to1}\hilb\left(x,x^t,qt(1-x)-1\right)=\displaystyle\prod\limits_{i=1}\limits^r \left(qt+e_i^N(V) t + e_i^G(\VN)\right).$
\end{corollary}

\subsection{Sum Side}
\label{sec:sum}
By taking traces, we now compute a formula for the Poincar\'e series $\mathcal{P}(x,y,u)$ defined in \Cref{eq:p-series-defn} as a sum over elements of $G$. To simplify notation, we denote by $\VN_m$ the homogeneous component of $\VN$ corresponding to the dual of $\VN^*_m$.

\begin{lemma} \label{thm:5.2} 
$
\hilb(x,y,u)=
\frac{1}{|G|} \displaystyle\sum\limits_{g \in G} \frac{\prod_{m\in\ENs}\mathrm{det}(1+uy^{m}g|_{\VN_m})}{\det(1-xg|_{V})}.$
\end{lemma}

\begin{proof}
Since $\UNsd\simeq \bigoplus_{m\in\ENs} \VN^*_m$, we have that
$\bigwedge \UNsd\simeq \bigotimes_{m\in\ENs}\bigwedge \VN^*_m$ as $G$-modules. hence, for each $g\in G$, \[\sum_{m,p\geq 0} \mathrm{tr}(g|({\textstyle\bigwedge^p}\UNsd)_{m})y^mu^p= \prod_{m\in\EN}\left(\sum_{p\geq 0}\mathrm{tr}(g|{\textstyle\bigwedge^p} \VN^*_m )y^{pm}u^p\right).\]

For each $m\in\ENs$ we have $\displaystyle\sum\limits_{p\geq 0}\left(\mathrm{tr}(g|{\textstyle\bigwedge^p}\VN^*_m)y^{pm}u^p\right)=\mathrm{det}(1+y^{m}ug|_{\VN^*_{m}}).$
 Therefore,
\begin{align*}
\hilb(x,y,u)&=\left(\sum_{\ell\geq 0}\mathrm{tr}(g|(S(\Vd)_\ell)x^\ell\right)\left(\sum_{m,p\geq 0}\mathrm{tr}(g|({\textstyle \bigwedge^p}\UNsd)_m)y^mu^p\right)\\
&=\frac{\prod_{m\in\ENs}\mathrm{det}(1+uy^{m}g^{-1}|_{\UNs_m})}{\det(1-xg^{-1}|_{V})}.\end{align*}
The result follows after taking the average over $G$ on each side.\qedhere
\end{proof}

The story is not quite so simple as just setting the sum over $G$ from~\Cref{thm:5.2} equal to the product, and then specializing.  The trouble is that in this specialized sum over $G$ from~\Cref{thm:5.2}, each element of $G$ \emph{does not} necessarily contribute the ``correct amount'' specified by the sum side of~\Cref{thm:main_theorem}---in particular, $(ng)|_{\VN}$ often has larger fixed space than $(ng)|_{\V}$, which causes many terms in the term-by-term limit to be zero.  It turns out, as we will now show, that the contributions \emph{are} correct when taken coset-by-coset.

\subsection{Sum Side, Coset-by-Coset I}
\label{sec:cosetii}
Fix some $g \in G$.  We find a product formula for \Cref{thm:5.2} restricted to the coset $gN$. Define \[\hilb_{gN}(x,y,u):=\frac{1}{|N|}\sum_{n \in N} \frac{\prod_{m\in\ENs} \det\left(1+u y^{m} (ng)|_{\VN^*_{m}}\right) }{\det(1-x(ng)|_{V^*})}.\]
Given $g\in G$, we can choose the fundamental $N$-invariants $\n_1,\dots,\n_r\in S(V^*)^G$ to also form a $g$-eigenbasis for $\VNd=\spn_\CC\{\n_1,\dots,\n_r\}$, since this space is $G$-stable and $g$ has finite order.  For $g \in G$, let $\epsilon_1^g,\dots,\epsilon_r^g$ denote the eigenvalues of $g$ on $\VNd$, so that $g\n_i=\epsilon_i^g\n_i$. 

\begin{proposition}
\label{prop:twisted_graded}
$\hilb_{gN}(x,y,u) = \displaystyle\prod\limits_{i=1}^r \frac{1+ \epsilon_i^g u y^{e_i^N(V)}}{1-\epsilon_i^g x^{d_i^N}}.$
\end{proposition}

\begin{proof}
First, $\displaystyle\prod\limits_{m\in\ENs}\mathrm{det}(1+uy^{m}(ng)|_{\VN^*_{m}})=\displaystyle\prod\limits_{i=1}\limits^r(1+\epsilon_i^guy^{e_i^N(\Vs)})$ uniformly for any $n\in N$, since $\VN^*_m$ is $N$-invariant. It remains to show that \[{\frac{1}{|N|}\sum_{n\in N}\frac{1}{\mathrm{det}(1-x(ng)|_{\Vd})}=\prod_{i=1}^r\frac{1}{1-\epsilon_i^gx^{d_i^N}}}.\]

Since $S(\VN^*)\simeq\bigotimes_{m\in\mathcal{E}_N}\mathrm{Sym}(\VN^*_m),$ where $\VN^*_m$ denotes the span of fundamental $N$-invariants having $\mathbf{x}$-degree $m+1$ and $\mathrm{Sym}(\VN^*_m)$ denotes its symmetric algebra, we have
\[\sum_{\ell\geq 0}\mathrm{tr}\bigl(g|S(\VN^*)_\ell\bigr)x^\ell = \prod_{m\in\mathcal{E}_N}\Bigl(\sum_{\ell\geq 0}\bigl(\mathrm{tr}(g|\mathrm{Sym}^\ell(\VN^*_m))(x^{m+1})^\ell\bigr), \] where $S(\VN^*)_\ell:=S(\VN^*)\cap S(V^*)_\ell$ and $S(V^*)_\ell$ as before denotes the homogeneous subspace of polynomials of $\mathbf{x}$-degree $\ell$. On the other hand,
\[\prod_{m\in\mathcal{E}_N}\Bigl(\sum_{\ell\geq 0}\mathrm{tr}\bigl(g|\mathrm{Sym}^\ell(\VN^*_{m})\bigr)(x^{m+1})^\ell\Bigr)=\prod_{m\in\mathcal{E}_N}\frac{1}{\mathrm{det}(1-x^{m+1}(g|_{\VN^*_{m}}))}=\prod_{i=1}^r\frac{1}{1-\epsilon_i^gx_i^{d_i^N}},\]
since $d_i^N=e_i^N(V)$. Therefore, $\displaystyle\sum\limits_{\ell\geq 0}\mathrm{tr}\bigl(g|S(\VN^*)_\ell\bigr)x^\ell=\displaystyle\prod\limits_{i=1}^r\frac{1}{1-\epsilon_i^gx_i^{d_i^N}}.$ Since for each $n\in N$ we have $\displaystyle\sum\limits_{\ell\geq 0}\mathrm{tr}\bigl(ng|S(V^*)_\ell\bigr)x^\ell = \frac{1}{\mathrm{det}(1-x(ng|_{V^*}))},$ it remains to show that, for each $\ell\geq 0$, \[\frac{1}{|N|}\sum_{n\in N}\Bigl(\mathrm{tr}\bigl(ng|S(V^*)_\ell\bigr)\Bigr)=\mathrm{tr}\bigl(g|S(\VN^*)_\ell\bigr).\]
To see this, note that the operator $\frac{1}{|N|}\displaystyle\sum\limits_{n\in N}ng=g\cdot\left(\frac{1}{|N|}\displaystyle\sum\limits_{n\in N}n\right)=g\circ\mathrm{pr}^N_\ell,$ where $\mathrm{pr}^N_\ell=\frac{1}{|N|}\sum_{n\in N} n$ is the projection from $S(V^*)_\ell$ onto its $g$-stable subspace $S(V^*)_\ell^N=S(E^*)_\ell$, whence $\mathrm{tr}\bigl((g\circ\mathrm{pr}^N_\ell)|S(V^*)_\ell\bigr)=\mathrm{tr}\bigl(g|S(E^*)_\ell\bigr)$. \end{proof}

\subsection{Sum Side, Coset-by-Coset II}
\label{sec:coseti}
We next specialize some results of \cite{bonnafe2006twisted} to the case when $N$ is a normal reflection subgroup of a complex reflection group $G$. They consider the more general situation when $N$ is translated by an arbitrary element in the normalizer of $N$ in $\mathrm{GL}(V)$.  Continue to fix some $g \in G$. 

\begin{proposition}[{\cite{bonnafe2006twisted}}]
\label{thm:twisted_os}
\[\frac{1}{|N|}\displaystyle\sum\limits_{n \in N} \frac{\det(1+u (ng)|_\Vd)}{\det(1-x (ng)|_\Vd)} = \displaystyle\prod\limits_{i=1}^r \frac{1+ \epsilon_i^g u x^{e_i^N(V)}}{1-\epsilon_i^g x^{d_i^N}} \left( = \hilb_{gN}(x,x,u)\right).\]
\end{proposition}

Specializing both sides of~\Cref{thm:twisted_os} to $u=q(1-x)-1$ and then taking the limit $x \to 1$ yields the following simple formula for sums over cosets.

\begin{corollary}[{\cite{bonnafe2006twisted}}]
\label{cor:twisted_os} 
\[\displaystyle\sum\limits_{n \in N} q^{\fix_V(ng)} = \left(\displaystyle\prod\limits_{\epsilon^g_i =1} q+e_i^N(V)\right) \left(\displaystyle\prod\limits_{\epsilon_i^g \neq 1} d_i^N\right).\]
\end{corollary}

Using~\Cref{cor:twisted_os}, we obtain the following crucial specialization of \Cref{prop:twisted_graded}, exploiting the fact that \Cref{prop:twisted_graded} gives the series $\hilb_{gN}(x,y,u)$, while \Cref{thm:twisted_os} gives the series $\hilb_{gN}(x,x,u)$.

\begin{corollary}
\label{cor:twisted_graded}
$\lim_{x\to 1} \hilb_{gN}(x,x^t,qt(1-x)-1) = t^{\fix_E g}\sum_{n \in N} q^{\fix_V (ng)}.$
\end{corollary}

\subsection{Proof of Theorem \ref{thm:main_theorem}}  We now prove our main theorem.
\begin{proof}[Proof of~\Cref{thm:main_theorem}] 
  Equating the formulas from~\Cref{lem:product_side,thm:5.2} gives
\begin{align}\label{eq:sum_eq_product}
\hilb(x,y,u)=\sum_{g \in G} \frac{\prod_{m\in\ENs}\mathrm{det}(1+uy^{m}g|_{\VN_m})}{\det(1-xg|_{V})} = |G|\prod_{i=1}^r\frac{1+x^{e_i^G(\UNs)}y^{e_i^N(V)}u}{1-x_i^{d_i^G}}. \end{align}

Let $\{g_j\}_{j=1}^{|H|}$ be a set of coset representatives for $N$ in $G$. By~\Cref{cor:twisted_graded}, splitting the sum side of~\Cref{eq:sum_eq_product} into a sum over the cosets of $N$ and specializing gives
\begin{align*}
\lim_{x\to 1} \hilb(x,x^t,qt(1-x)-1)&= \sum_{j=1}^{|H|}\lim_{x\to 1} \hilb_{gN}(x,x^t,qt(1-x)-1) \\ &= \sum_{j=1}^{|H|} t^{\fix_{E} g_j}\sum_{n \in N} q^{\fix_V ng_j} = \sum_{g \in G} q^{\fix_V g} t^{\fix_{E} g}.
\end{align*}

The result now follows from~\Cref{eq:sum_eq_product} by equating this specialization of the sum side with the same specialization of the product side from~\Cref{cor:prod_side_specialization}.
\end{proof}

\section{Classification of Normal Reflection Subgroups}
\label{sec:classification}
In the interest of space, we restrict our classification of normal reflection subgroups to $\mathrm{rank} \geq 3$.  In rank 2, there are two connected posets of imprimitive complex reflection groups ordered by normality: one has maximal element $G_{11}$ and minimal elements $G(4,2,2)$, $G_4$, and $G_{12}$, while the other has maximal element $G_{19}$ and minimal elements $G_{16}$, $G_{20}$, and $G_{22}$.

\begin{theorem}[{\cite[Corollary 2.18]{lehrer2009unitary}}]
For $r\geq 3$, the normal reflection subgroups of $G(ab,b,r)$ are $(C_d)^r$ and  $G(ab,db,r)$ for $d|a$, giving quotients $G(ab,b,r)/(C_d)^r = G((a/d) b,b,r)$ and $G(ab,b,r)/G(ab,db,r) = C_d$.
\end{theorem}

As $G_{26}$ and $G_{28}$ are the only exceptional reflection groups with more than a single orbit of reflections, there are three nontrivial exceptional (that is, not imprimitive) examples of normal reflection subgroups in rank greater than two: $G(3,3,3) \triangleleft G_{26}$, with quotient $G_4$; $G_{25} \triangleleft G_{26}$, with quotient $C_2$;  and $G(2,2,4) \triangleleft G_{28} \simeq W(F_4)$, with quotient $\mathfrak{S}_3$.

\section{Reflexponents}
\label{sec:reflexponents}
Fix $G$ a complex reflection group of rank $r$ with reflection representation $V$.  Call an $r$-dimensional representation $M$ of $G$ \defn{factorizing} if $M$ has dimension $r$ and
\[\sum_{g \in G} q^{\mathrm{fix}_V(g)} t^{\mathrm{fix}_M(g)} = \prod_{i=1}^r \Big(qt+(e_i^G(V)-m_i)t+ m_i\Big),\] for some nonnegative integers $m_1,\ldots,m_r$.  More generally, call a representation $M$ of $G$ of dimension $\dim M \leq r$ factorizing if it is factorizing in the above sense after adding in $r-\dim M$ copies of the trivial representation.

We can now give a uniform explanation for certain ad-hoc identities from~\cite{williams2019reflexponents}.  Let $\mathcal{H}$ be an orbit of reflecting hyperplanes, write $\mathcal{R}_\mathcal{H}$ for the set of reflections fixing some $L \in \mathcal{H}$, and let $N_\mathcal{H} = \left\langle \mathcal{R}_\mathcal{H} \right\rangle$ be the subgroup generated by reflections around hyperplanes in $\mathcal{H}$.  Since these reflections form a conjugacy class in $G$, $N_\mathcal{H}$ is a normal reflection subgroup of $G$. Furthermore: the quotient $G/N_\mathcal{H}$ acts as a reflection group on the $N_\mathcal{H}$-invariants of $V$; and this action gives a $G$-representation $M_\mathcal{H}$ that is factorizing by~\Cref{thm:main_theorem}.

\section{Galois Twists}\label{sec:galois}

Let $V$ be an $r$-dimensional complex vector space and $G\subset\mathrm{GL}(V)$ be a complex reflection group. It is known that $G$ can be realized over $\mathbb{Q}(\zeta_G)$, where $\zeta_G$ is a primitive $|G|$-th root of unity, in the sense that there is a basis for $V$ with respect to which $G\subset \mathrm{GL}_r(\mathbb{Q}(\zeta_G))$. For $\sigma\in\mathrm{Gal}(\mathbb{Q}(\zeta_G)/\mathbb{Q})$, the \defn{Galois twist} $V^\sigma$ is the representation of $G$ on the same underlying vector space $V$ obtained by applying $\sigma$ to the matrix entries of $g\in\mathrm{GL}(\mathbb{Q}(\zeta_G))$. Orlik and Solomon found a beautiful generalization of \Cref{eq:shephard_todd} that takes into account these Galois twists. Below we write $\lambda_1(g),\dots,\lambda_r(g)$ for the eigenvalues of $g\in G$ on $V$.

\begin{theorem}[{\cite{orlik1980unitary}}]
\label{thm:orlik_solomon}
Fix $G$ a reflection group and $\sigma \in \mathrm{Gal}(\mathbb{Q}(\zeta_G)/\mathbb{Q})$.  Then 
\[\sum_{g \in G} \left(\prod_{\lambda_i(g) \neq 1} \frac{1-\lambda_i(g)^\sigma}{1-\lambda_i(g)}\right) q^{\fix_V g} = \prod_{i=1}^r \left(q+e_i(V^\sigma) \right).\]
\end{theorem}

\begin{definition}
Let $I^G_+\subset S(V^*)$ be the ideal generated by homogeneous $G$-invariant polynomials of positive degree, and let $\mathcal{C}_G$ be a $G$-stable homogeneous subspace of $S(V^*)$ such that $S(V^*)\simeq I^G_+\oplus \mathcal{C}_G$ as $G$-modules. For $\sigma\in\mathrm{Gal}(\zeta_G)$, define the \defn{Orlik-Solomon space} $U_G^\sigma:=(\mathcal{C}_G\otimes (\V^\sigma)^*)^G=\spn_\CC\{u_1^G,\dots,u_r^G\}$, where the $u_i^G$ are homogeneous with $\mathrm{deg}(u_i^G)=e_i^G(V^\sigma)$ and such that $\bigl(S(V^*)\otimes (V^\sigma)^*\bigr)^G\simeq S(V^*)^G\otimes U_G^\sigma$.
\label{def:osspace}
\end{definition}

 We can now state the general form of our main theorem.

\renewcommand{\thetheorem}{\ref{thm:main_theorem}}
\begin{theorem}
Let $N\triangleleft \ G$ be reflection groups acting by reflections on $V$, and let $\VN=V/N$. Let $\sigma\in\mathrm{Gal}(\mathbb{Q}(\zeta_G)/\mathbb{Q})$ and define the Orlik-Solomon space $U_N^\sigma$ as in \Cref{def:osspace}. Then
\[\sum_{g \in G}\left( \prod_{\lambda_i(g) \neq 1} \frac{1-\lambda_i(g)^\sigma}{1-\lambda_i(g)}\right) q^{\fix_V g} t^{\fix_{E} g} = \prod_{i=1}^r \left(qt+e_i^N(V^\sigma) t + e_i^G((U_N^\sigma)^*)\right).\]
\end{theorem}
\addtocounter{theorem}{-1}

\begin{remark}
The Orlik-Solomon space $U_N^\sigma$ playes the role of $\VN^*$ in this generalization of~\Cref{thm:main_theorem}. When $\sigma=1$, a straightforward argument yields a graded $G$-module homomorphism (of degree $-1$) $\VNd\simeq U_N$. This is why we shifted the obvious $\mathbf{x}$-grading of $\VN^*$ by $-1$ in \Cref{sec:poincare}. However, it is not true in general that $(\VN^\sigma)^*\simeq U_N^\sigma$ when $\sigma\neq 1$.
\end{remark}

\providecommand{\bysame}{\leavevmode\hbox to3em{\hrulefill}\thinspace}
\providecommand{\MR}{\relax\ifhmode\unskip\space\fi MR }
\providecommand{\MRhref}[2]{%
  \href{http://www.ams.org/mathscinet-getitem?mr=#1}{#2}
}
\providecommand{\href}[2]{#2}

\end{document}